\documentclass[12pt]{amsart}
\usepackage[a4paper,hmargin=3.5cm,vmargin=4cm]{geometry}
\usepackage{graphicx} 
\usepackage{verbatim,enumerate}
\usepackage[dvipsnames]{xcolor}
\usepackage[colorlinks]{hyperref}
\usepackage{ulem}
\usepackage {svg}
\usepackage{wrapfig}
\usepackage{amsrefs}
\usepackage{amssymb, amsmath}
\usepackage{bigints}
\usepackage{here}

\usepackage{pb-diagram}

\title[Reflection principles for zero mean curvature surfaces in $\mathbb{I}^3$]{%
Reflection principles for zero mean curvature surfaces in the simply isotropic $3$-space
}

\author[S.~Akamine and H.~Fujino]{
Shintaro Akamine and Hiroki Fujino
}   

\address[Shintaro Akamine]{%
Department of Liberal Arts, College of Bioresource Sciences,
Nihon University, 
1866 Kameino, Fujisawa, Kanagawa, 252-0880, Japan}
\email{akamine.shintaro@nihon-u.ac.jp}

\address[Hiroki Fujino]{%
Institute for Advanced Research, Graduate School of Mathematics, 
Nagoya University, Furo-cho, Chikusa-ku, Nagoya 464-8602, Japan
}
\email{m12040w@math.nagoya-u.ac.jp}



\thanks{
The first author was partially supported by 
JSPS KAKENHI Grant Number 19K14527,
and the second author by JSPS KAKENHI Grant Number 20K14306.
}

\usepackage{amsthm}
\theoremstyle{plain}

 \newtheorem{theorem}{Theorem}[section]
 \newtheorem{proposition}[theorem]{Proposition}

 \newtheorem{lemma}[theorem]{Lemma}
 \newtheorem{corollary}[theorem]{Corollary}

\theoremstyle{definition}
 
\theoremstyle{remark}
 \newtheorem{remark}[theorem]{Remark}
 \newtheorem*{remark*}{Remark}
\newtheorem{example}[theorem]{Example}
 \newtheorem*{acknowledgement}{Acknowledgement}

\numberwithin{equation}{section}

\newcommand{\vect}[1]{\boldsymbol{#1}}

\renewcommand{\phi}{\varphi}

\definecolor{Blue}{rgb}{0,0,1}  
\definecolor{Red}{rgb}{1,0,0}  

\begin{document}
\maketitle

\begin{abstract}
Zero mean curvature surfaces in the simply isotropic 3-space $\mathbb{I}^3$ naturally appear as  intermediate geometry between geometry of minimal surfaces in $\mathbb{E}^3$ and that of maximal surfaces in $\mathbb{L}^3$.
In this paper, we investigate reflection principles for zero mean curvature surfaces in $\mathbb{I}^3$ as with the above surfaces in $\mathbb{E}^3$ and $\mathbb{L}^3$. In particular, we show a reflection principle for isotropic line segments on such zero mean curvature surfaces in $\mathbb{I}^3$, along which the induced metrics become singular.
\end{abstract}


\section{Introduction} \label{sec:1} 
Recently, there has been a growing interest for surfaces in the isotropic $3$-space $\mathbb{I}^3$, which is the 3-dimensional vector space $\mathbb{R}^3$ with the degenerate metric $dx^2+dy^2$. Here, $(x,y,t)$ are canonical coordinates on $\mathbb{I}^3$. In particular, a class of surfaces in $\mathbb{I}^3$ naturally appears as an intermediate geometry between geometry of minimal surfaces in the Euclidean $3$-space $\mathbb{E}^3$ and that of maximal surfaces in the Loretnz-Minkowski $3$-space $\mathbb{L}^3$. In fact, if we consider the deformation family introduced in \cite{AF3} with parameter $c\in \mathbb{R}$
\begin{equation}\label{eq:Wformula}
	X_{c}(w)=\mathrm{Re}\int^w \left(1-c G^2,-i(1+c G^2),2G\right)Fd\zeta,\quad w\in D, 
\end{equation} 
on a simply connected domain $D\subset \mathbb{C}$. Here, the pair $(F,G)$ of a holomorphic function $F$ and a meromorphic function $G$ on $D$ is called a {\it Weierstrass data} of $X_c$. Interestingly, \eqref{eq:Wformula} represents Weierstrass-type  formulae for minimal surfaces in $\mathbb{E}^3$ when $c=1$ and for maximal surfaces in $\mathbb{L}^3$ when $c=-1$. When we take $c=0$, \eqref{eq:Wformula} is nothing but the representation formula for zero mean curvature surfaces in $\mathbb{I}^3$, see \cite{AF3,Si,MaEtal,SY,Sato,Pember} for example and Figure \ref{Fig:Cats}. 

\begin{figure}[htbp]
 \vspace{3.0cm}
\hspace*{-11ex}
    \begin{tabular}{cc}
      \begin{minipage}[t]{0.55\hsize}
        \centering
                \vspace{-3.2cm}
        \includegraphics[keepaspectratio, scale=0.5]{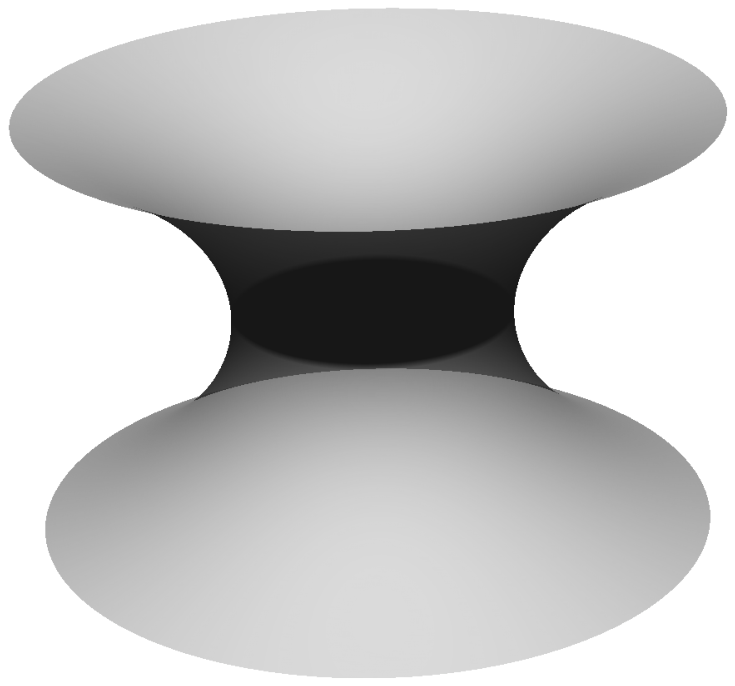}
              \end{minipage} 
                 \hspace*{-18ex}
        \begin{minipage}[t]{0.55\hsize}
        \centering
            \vspace{-3.3cm}
        \includegraphics[keepaspectratio, scale=0.5]{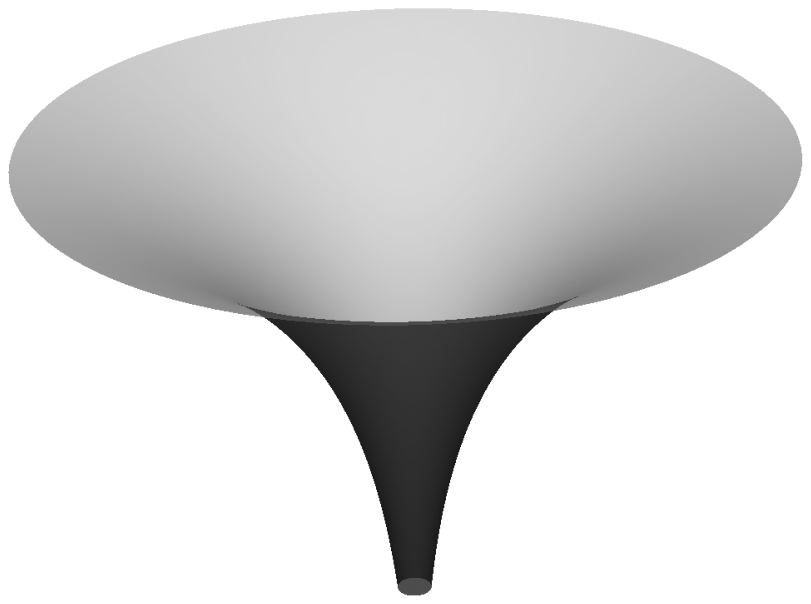}
              \end{minipage} 
           \hspace*{-15ex}
      \begin{minipage}[t]{0.50\hsize}
        \centering
        \vspace{-3.1cm}
        \includegraphics[keepaspectratio, scale=0.55]{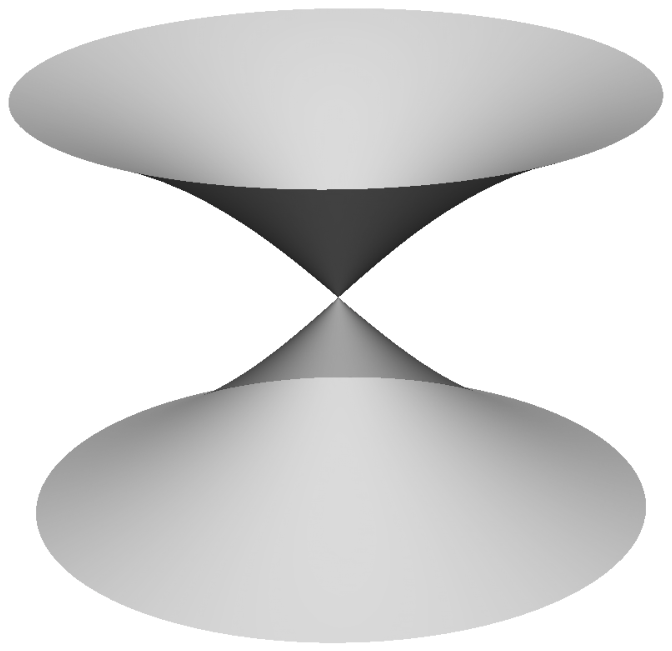}
            \end{minipage} 
    \end{tabular}
      \caption{The catenoid $X_1$ in $\mathbb{E}^3$ (left), isotropic catenoid $X_0$ in $\mathbb{I}^3$ (center) and elliptic catenoid $X_{-1}$ in $\mathbb{L}^3$ (right). These surfaces are related by one deformation family $\{X_c\}_{c\in \mathbb{R}}$.}
       \label{Fig:Cats}
  \end{figure}

One thing that is distinctive about surfaces in $\mathbb{I}^3$ is 
concerning vertical lines. Since the metric in $\mathbb{I}^3$ ignores the vertical component of vectors, the induced metric on a surface degenerates on vertical lines in $\mathbb{I}^3$. In this sense, vertical lines in $\mathbb{I}^3$ are intriguing and important objects in geometry in $\mathbb{I}^3$, and such a vertical line in $\mathbb{I}^3$ is called an {\it isotropic line}.

In this paper, we solve the problem raised in the paper by Seo-Yang \cite[Remark 30]{SY}, which we can state the following:
\vspace{0.3cm}
\begin{center}
{\it Is there a principle of analytic continuation across isotropic lines\\ on zero mean curvature surfaces in $\mathbb{I}^3$?}
\end{center}
\vspace{0.3cm}
Obviously, it is directly related to a reflection principle for zero mean curvature surfaces in $\mathbb{I}^3$.
  
 The main theorem of this paper is as follows (see also Figure \ref{Fig:Assumption}).  
 
\begin{theorem}\label{thm:reflection_Intro}
Let $S\subset \mathbb{I}^3$ be a bounded zero mean curvature graph over a simply connected Jordan domain $\Omega \subset \mathbb{C}\simeq \text{${xy}$-plane}$. If $\partial{S}$ has an isotropic line segment $L$ on a boundary point $z_0\in \partial{\Omega}$ satisfying
\begin{itemize}
\item  $L$ connects two horizontal curves $\gamma_1$ and $\gamma_2$ on $\partial{S}$, and 
\item the projections of $\gamma_1$ and $\gamma_2$ into the $xy$-plane form a regular analytic curve near $z_0$. We denote this analytic curve on the $xy$-plane by $\Gamma$.
\end{itemize}
 Then $S$ can be extended real analytically across $L$ via the analytic continuation across $\Gamma$ in the $xy$-plane and the reflection with respect to the height of the midpoint of $L$ in the $t$-direction.
\end{theorem}

  \begin{figure}[htbp]
 \vspace{3.3cm}
\hspace{-8ex}
    \begin{tabular}{cc}
      \begin{minipage}[t]{0.55\hsize}
        \centering
                \vspace{-3.5cm}
                \hspace*{-12ex}
        \includegraphics[keepaspectratio, scale=0.55]{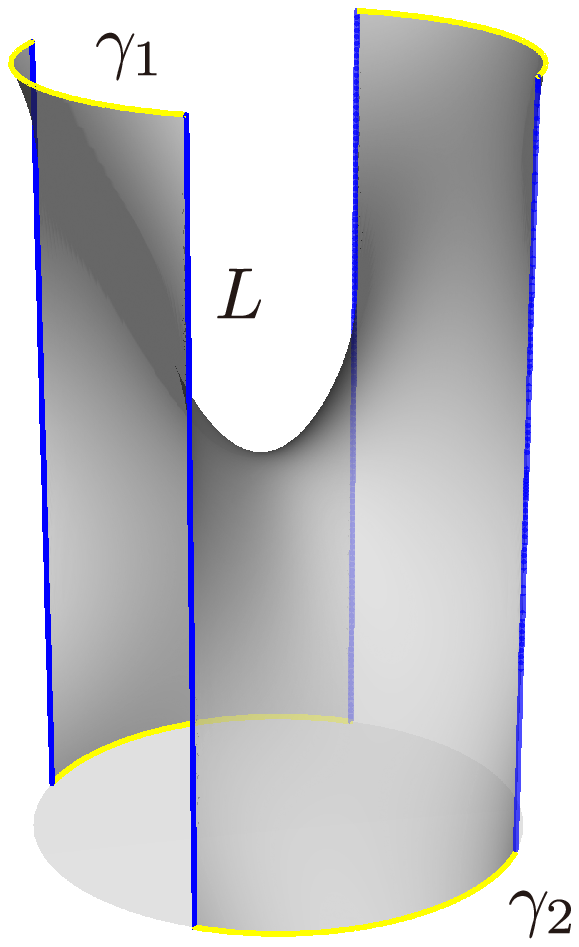}
              \end{minipage} 
                 \hspace{-7ex}
        \begin{minipage}[t]{0.55\hsize}
        \centering
            \vspace{-3.5cm}
        \includegraphics[keepaspectratio, scale=0.65]{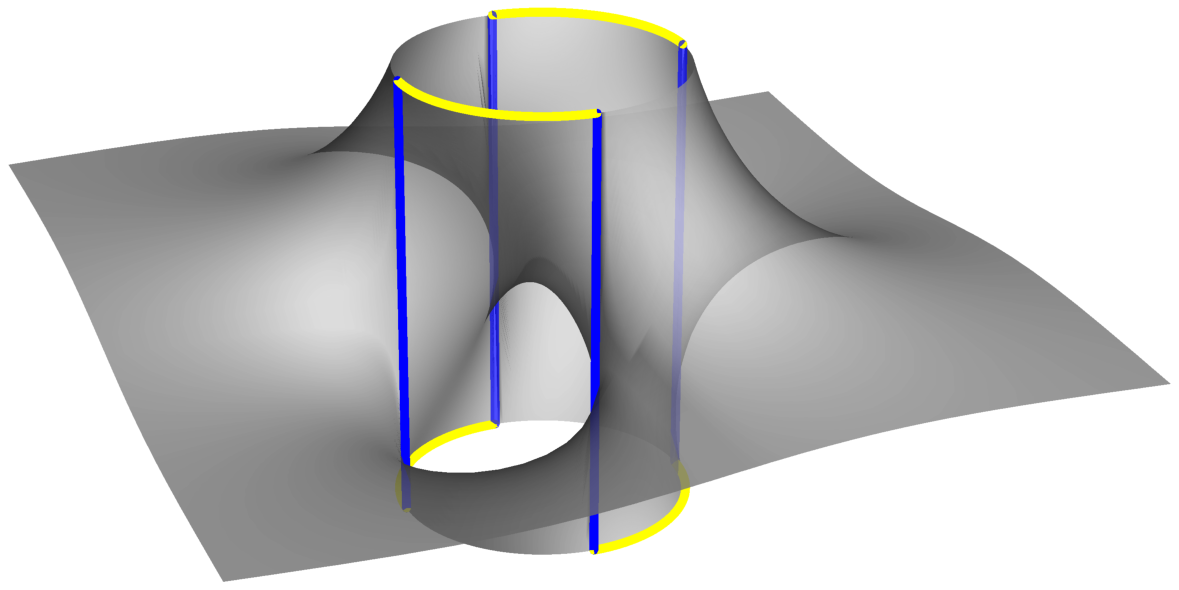}
              \end{minipage} 
    \end{tabular}
     \vspace{-0.5cm}
      \caption{A surface with the boundary condition assumed in Theorem \ref{thm:reflection_Intro} (left) and its analytic extension (right). Each vertical line (blue one) indicates $L$ and horizontal curves connected by $L$ (yellow ones) indicate $\gamma_1$ and $\gamma_2$ in Theorem \ref{thm:reflection_Intro}.}
       \label{Fig:Assumption}
  \end{figure}
Moreover, we also investigate how this kind of isotropic lines appear on boundary of zero mean curvature surfaces in $\mathbb{I}^3$, see Theorem \ref{thm:reflection} for more details.

As an application of Theorem \ref{thm:reflection_Intro}, we give an analytic continuation of some examples including the helicoid across isotropic lines. We also give triply periodic zero mean curvature surfaces with isotropic lines. One of them is analogous to the Schwarz' D-minimal surface in $\mathbb{E}^3$.

The organization of this paper is as follows. In Section 2, we give a short summary of zero mean curvature surfaces in  $\mathbb{I}^3$. In Section 3, we first prove a reflection principle for continuous boundaries of zero mean curvature surfaces in $\mathbb{I}^3$ as the classical reflection principle for minimal surfaces in $\mathbb{E}^3$. After that we give a proof of Theorem \ref{thm:reflection_Intro}. 
Finally, we give some examples of zero mean curvature surfaces with isotropic lines in Section 4.


\section{Preliminary} \label{sec:2} 

In this section, we recall some basic notions of zero mean curvature surfaces in  $\mathbb{I}^3$. See \cite{Sachs,Sato, Strubecker, Strubecker2, SY, Si} and their references for more details.

The {\it simply isotropic $3$-space} $\mathbb{I}^3$ is the 3-dimensional vector space $\mathbb{R}^3$ with the degenerate metric $\langle \ ,\ \rangle := dx^2+dy^2$. 
Here, $(x,y,t)$ are canonical coordinates on $\mathbb{I}^3$. 
A non-degenerate surface in $\mathbb{I}^3$ is an immersion $X\colon D \to \mathbb{I}^3$ from a domain $D\subset \mathbb{C}$ into $\mathbb{I}^3$ whose induced metric $ds^2:=X^*\langle\ ,\ \rangle$ is positive-definite. Then, we can define the Laplacian $\Delta$ on the surface with respect to $ds^2$. 

Since minimal surfaces in the Euclidean $3$-space $\mathbb{E}^3$ and maximal surfaces in the Minkowski $3$-space $\mathbb{L}^3$ are characterized by the equation 
\begin{equation}\label{eq:harmonic}
\Delta{X}= \left( \Delta{x}, \Delta{y}, \Delta{t} \right) \equiv \vect{0},
\end{equation} 
we can also consider surfaces in $\mathbb{I}^3$ satisfying the equation \eqref{eq:harmonic}. Such a surface is called a {\it zero mean curvature surface} or {\it isotropic minimal surface} in $\mathbb{I}^3$. 

Since $ds^2$ is positive-definite, each tangent plane of the surface $X$ is not vertical. Hence, each surface $X$ is locally parameterized by $X(x,y)=(x,y,f(x,y))$ for a function $f=f(x,y)$ and the Laplacian $\Delta$ can be written as $\Delta = \partial_x^2+\partial_y^2$ on this coordinates. This means that a zero mean curvature surface in $\mathbb{I}^3$ is locally the graph of a harmonic function on the $xy$-plane.

As well as minimal surfaces in $\mathbb{E}^3$ and maximal surfaces in $\mathbb{L}^3$, we can see by the equation \eqref{eq:harmonic} that zero mean curvature surfaces in $\mathbb{I}^3$ also have the following Weierstrass-type representation formula.

\begin{proposition}[cf.~\cite{MaEtal,Pember,Sato,Strubecker,SY}]\label{prop:W}
Let $F\not \equiv 0$ be a holomorphic function on a simply connected domain $D\subset \mathbb{C}$ and $G$ a meromorphic function on $D$ such that $FG$ is holomorphic on $D$. Then the mapping 
\begin{equation}\label{eq:w-formula}
X(w)=\mathrm{Re}\int^w {}^t(1,-i,2G)Fd\zeta
\end{equation}
gives a zero mean curvature surface in $\mathbb{I}^3$. Conversely, any zero mean curvature surface in $\mathbb{I}^3$ is of the form \eqref{eq:w-formula}. 
\end{proposition}
The pair $(F,G)$ is called a {\it Weierstrass data} of $X$.

\begin{remark}
Since the induced metric is $ds^2=\lvert F \rvert^2dwd\overline{w}$, we can consider zero mean curvature surfaces in $\mathbb{I}^3$ with singular points by the equation \eqref{eq:w-formula}. Singular points, on which the metric $ds^2$ degenerates, correspond to isolated zeros of the holomorphic function $F$. 
\end{remark}
At the end of this section, we mention another representation by using harmonic functions. 
Let us define the holomorphic function $h(w)=\int^w Fd\zeta$ on $D$. Then the conformal parametrization of $X(w)=(x(w),y(w),t(w))$ in \eqref{eq:w-formula} is written as $X(w)=(h(w),t(w))$, here we identify the $xy$-plane with the complex plane $\mathbb{C}$.

\section{Main theorem} \label{sec:3} 
\subsection{Reflection principle for horizontal curves}

As in the classical minimal surface theory in $\mathbb{E}^3$, the Schwarz reflection principle also leads to a symmetry principle for continuous planar boundary of zero mean curvature surfaces $\mathbb{I}^3$. In this subsection, we recall a reflection principle of this typical type.

A subset $\Gamma \subset \mathbb{C}$ is said to be a {\it regular simple analytic arc} if there exists an open interval $I\subset \mathbb{R}$ and an injective real analytic curve $\gamma\colon I \to \mathbb{C}$ such that $\gamma' \neq 0$ and $\gamma(I)=\Gamma$. 
We denote the analytic continuation of $\gamma$ into a neighborhood of $I$ by $\gamma$ again, and note that $\gamma$ is a conformal mapping around $I$ since $\gamma' \neq 0$. 
 We define the reflection map $R_{\Gamma}$ with respect to $\Gamma$ by the relation $R_\Gamma= R_\gamma:=\gamma \circ R \circ \gamma^{-1}$, where $R$ is the complex conjugation (note that we can easily see that $R_\Gamma=R_\gamma$ is independent of the choice of $\gamma$). We call this reflection $R_\Gamma$ the {\it reflection with respect to $\Gamma$}.

 Let $f\colon \mathbb{H} \to \mathbb{C}$ be a holomorphic function which extends continuously to an interval $J=(a,b)\subset \partial \mathbb{H}$. Here, $\mathbb{H}$ denotes the upper half-plane in $\mathbb{C}$. In this setting, let us recall the following reflection principle for holomorphic functions (the proof is given in the similar way to the discussion in \cite[Chapter 6, Section 1.4]{Ahl1}).

\begin{lemma} \label{lem:ref_in_analytic_arc}
Under the above assumption,  if the image $\Gamma :=f(J)$ is a regular simple analytic arc, then $f$ extends holomorphically to an open subset containing $\mathbb{H}\cup J$ so that
\[
f(\overline{w})=R_{\Gamma} \circ  f(w).
\]
\end{lemma}

As a typical case, the following reflection principle for zero mean curvature surfaces in $\mathbb{I}^3$ holds.
\begin{proposition}\label{prop:reflection}
Let $X\colon \mathbb{H}\to S$ be a conformal parametrization of a zero mean curvature surface $S\subset \mathbb{I}^3$. If $X$ is continuous on an open interval $I\subset \partial{\mathbb{H}}$ and $\Gamma:=X(I)$ is a regular (simple) analytic arc 
on a horizontal plane $P\simeq \mathbb{C}$, then $S$ can be extended real analytically across $\Gamma$ via the reflection with respect to $\Gamma$ in $xy$-direction and the planar symmetry with respect to $P$.
\end{proposition}

We should mention that this result was essentially obtained by Strubecker \cite{Strubecker3}. Here, we give a short proof of this fact for the sake of completeness.

\begin{proof}
We may assume that $P$ is the $xy$-plane. By the assumption, $X=(x,y,t)$ satisfies
$t\equiv 0$ on $I$, and $f(I)=X(I)=\Gamma$ is a regular analytic arc. By the Schwarz reflection principle (see \cite[Chapter 4, Section 6.5]{Ahl1}), the harmonic function $t$ can be extended across $I\subset \partial{\mathbb{H}}$ so that $t(\overline{w})=-t(w)$. 
On the other hand, by Lemma \ref{lem:ref_in_analytic_arc}, $f$ is also extended across $I$.
Therefore, $X(w)=(f(w), t(w))$ can be defined across $I$ and satisfies $X(\overline{w})=(R_{\Gamma} \circ  f(w),-t(w))$, which is the desired symmetry.
\end{proof}

As a special case of Proposition \ref{prop:reflection}, we can consider the following specific boundary conditions.

\begin{corollary}\label{cor:reflection_nonvertical}
Under the same assumptions as in Proposition \ref{prop:reflection}, the following statements hold.
\begin{itemize}
\item[(i)] If $\Gamma$ is a straight line segment on a horizontal plane $P$, then $S$ can be extended via the $180^\circ$-degree rotation with respect to $\Gamma$.
\item[(ii)]  If $\Gamma$ is a circular arc on a horizontal plane $P$, then $S$ can be extended via the inversion of the circle in the $xy$-direction and the planar symmetry with respect to $P$.
\end{itemize}
\end{corollary}

\begin{remark}[Reflection for boundary curves on non-vertical planes]
Proposition \ref{prop:reflection} and Corollary \ref{cor:reflection_nonvertical} are also valid when $P$ is a general non-vertical plane as follows:
If $P$ is a non-vertical plane, we can write it by the equation $t=ax+by+c$ for some $a,b,c \in \mathbb{R}$. After taking the affine transformation 
\begin{equation}\label{eq:iso}
(x,y,t)\longmapsto (x,y,t-ax-by-c)
\end{equation}
preserving the metric $\langle \ ,\ \rangle$, $P$ becomes a horizontal plane. Hence we can apply the reflection properties as in Proposition \ref{prop:reflection} and Corollary \ref{cor:reflection_nonvertical}. 
We remark that the affine transformation \eqref{eq:iso} is not an isometry in $\mathbb{E}^3$ and hence symmetry changes slightly after this transformation. For instance, the symmetry in (i) of Corollary \ref{cor:reflection_nonvertical} is no longer the $180^\circ$-degree rotation with respect to a straight line after the inverse transformation of \eqref{eq:iso}.  The transformation \eqref{eq:iso} is one of congruent motions in $\mathbb{I}^3$. See \cite{Pottmann, Sachs, Strubecker3} for example.
\end{remark}

\subsection{Reflection principle for vertical lines}\label{subsec:vline}
The classical Schwarz reflection principle is for harmonic functions which are at least continuous on their boundaries. On the other hand, as discussed in the proof of Theorem 2.3 and Remark 2.4 in {\cite{AF2}}, each harmonic function with a discontinuous jump point at the boundary has also a real analytic continuation across the boundary after taking an appropriate blow-up as follows.

Let $\Pi\colon D^+:=\mathbb{R}_{>0}\times (0, \pi) \to \mathbb{H}$ be a homeomorphism defined by $\Pi(r, \theta)=re^{i \theta}$. By definition, $\Pi$ is real analytic on the wider domain $D:=\mathbb{R}\times (0, \pi)$.

\begin{proposition}[\cite{AF2}]\label{prop:DiscontiRef}
Let $f\colon \mathbb{H} \to \mathbb{R}$ be a bounded harmonic function which is continuous on $\mathbb{H} \cup (-\varepsilon, 0) \cup (0,\varepsilon)$ for some $\varepsilon >0$. If $f\equiv a$ on $(-\varepsilon, 0)$ and  $f\equiv b$ on $ (0,\varepsilon)$, then the real analytic map $f\circ \Pi$ on $D^+$ extends to $D$ real analytically satisfying the following conditions.
\begin{itemize}
\item[(i)] $f\circ \Pi(-r,\pi-\theta)+X\circ\Pi(r,\theta)=a+b$, and 
\item[(ii)] $f\circ\Pi(0,\theta)=a\cfrac{\theta}{\pi}+b\left(1-\cfrac{\theta}{\pi}\right)$.
\end{itemize}
\end{proposition}

\begin{remark}[Blow-up of discontinuous point $0\in \partial{\mathbb{H}}$]
The condition (ii) in Proposition \ref{prop:DiscontiRef} means that $f\circ\Pi(0,\theta)$ is the point which divides the line segment connecting $a$ and $b$ into two segments with lengths $(1-\theta/\pi)\colon \theta/\pi$.
\end{remark}

As pointed out in \cite[Remark 30]{SY}, vertical lines naturally appear on boundary of zero mean curvature surfaces in $\mathbb{I}^3$, along which each tangent vector $\vect{v}$ has zero length: $\langle \vect{v}, \vect{v} \rangle=0$. In this sense, a vertical line segment in $\mathbb{I}^3$ is different from any other non-vertical lines and it is called an {\it isotropic line} (cf.~\cite{Pottmann}). Obviously, we cannot apply the usual Schwarz reflection principle for such boundary lines because the conformal structure on such a surface in $\mathbb{I}^3$ breaks down on isotropic lines. By using Proposition \ref{prop:DiscontiRef}, we can investigate such isotropic lines and a reflection property along them as follows.

\begin{theorem}\label{thm:reflection}
Let $S\subset \mathbb{I}^3$ be a bounded zero mean curvature graph over a simply connected Jordan domain $\Omega \subset \mathbb{C}$. If $\partial{S}$ has a isotropic line segment $L$ on a boundary point $z_0\in \partial{\Omega}$ connecting two horizontal curves on $\partial{S}$ 
whose projections to the $xy$-plane form a regular (simple) analytic arc $\Gamma$ near $z_0$,
then the following properties hold.
\begin{itemize}
\item[(i)] $S$ can be extended real analytically across $L$ via the reflection with respect to $\Gamma$ in the $xy$-direction and  the reflection with respect to the height of the midpoint of $L$ in the $t$-direction.
 \item[(ii)] $L$ coincides with the cluster point set $C(X, w_0)$ of a conformal parametrization $X=(h,t)\colon \mathbb{H}\to S$ at $w_0$ in $\partial{\mathbb{H}}$ satisfying $h(w_0)=z_0$.
\end{itemize}
\end{theorem}

\noindent Here, $C(X,w_0)$ consists of the points $z$ so that $z=\lim_{w_n\to w_0} X(w_n)$ 
for some $w_n \in \mathbb{H}$. 

\begin{proof}
Let us consider a conformal parametrization $X=(h,t)\colon \mathbb{H}\to S$, where $h$ is a holomorphic function defined on $\mathbb{H}$ satisfying $h(\mathbb{H})=\Omega$. Without loss of generality, we may assume that $h(0)=z_0$.

Since $t$ is a bounded harmonic function, it can be written as the Poisson integral
\[
	t(\xi +i\eta)=P_{\hat{t}}(\xi +i\eta)=\frac{1}{\pi}\int_{-\infty}^{\infty}\frac{\eta}{(\xi-s)^2+\eta^2}\hat{t}(s)ds.
\]
of some measurable bounded function $\hat{t}$ such that $\hat{t}(x)=\lim_{y\to 0} t(x+iy)$ almost every $x\in \mathbb{R}$ (see \cite[Chapter 3]{Katz}, and see \cite[Chapter 4, Section 6.4]{Ahl1} for the Poisson integral on $\mathbb{H}$). By the assumption, $\hat{t}$ has a discontinuous jump point at $w=0$ and we may assume that $\hat{t}\equiv a$ on $(-\varepsilon, 0)$ and  $\hat{t} \equiv b$ on $ (0,\varepsilon)$ for some $\varepsilon>0$. By Proposition \ref{prop:DiscontiRef}, $t\circ \Pi \colon D^+\to \mathbb{R}$ extends to $D$ real analytically so that 
\begin{equation}\label{eq:t}
t\circ \Pi(-r,\pi-\theta)+t\circ\Pi(r,\theta)=a+b,\quad (r, \theta) \in D^+.
\end{equation}

Next, we consider the function $h$. By the Carath\'eodory Theorem (see \cite[Theorem 17.16]{Mil1}), $h$ can be extended to $h\colon \overline{\mathbb{H}} \to \overline{\Omega}$ homeomorphically. The assumption implies that 
$h({-\varepsilon, \varepsilon})$ is a regular analytic curve for some $\varepsilon >0$ and hence $h$ can be extended across $(-\varepsilon,\varepsilon)$ so that $h(\overline{w})=R_{\Gamma} \circ h(w)$ by Lemma \ref{lem:ref_in_analytic_arc}.
 Therefore $h\circ \Pi \colon D^+\to \Omega$ extends across $\{(0,\theta)\mid 0<\theta<\pi\}$ real analytically and it satisfies 
\begin{equation}\label{eq:h}
h\circ \Pi (-r,\pi-\theta) =h(\overline{re^{i\theta}}) =\left( R_{\Gamma} \circ h\circ \Pi \right) (r,\theta).
\end{equation}

By the equations \eqref{eq:t} and \eqref{eq:h}, $X\circ \Pi$ can be extended across $\{(0,\theta)\mid 0<\theta<\pi\}$ and satisfies
\[
X\circ \Pi (-r,\pi-\theta)=\left((R_{\Gamma} \circ h\circ \Pi) (r,\theta), a+b-t\circ\Pi(r,\theta)\right),
\]
which implies the desired reflection across $L$. 

Finally, by (ii) of Proposition \ref{prop:DiscontiRef}, we obtain the relation
\[
X\circ \Pi (0, \theta)=\left(h(0), t\circ \Pi (0, \theta)\right) = \left(z_0, a\cfrac{\theta}{\pi}+b\left(1-\cfrac{\theta}{\pi}\right) \right).
\]
This means that the cluster point set $C(X,0)$ of $X$ at $w_0=0$ becomes $L$. 
\end{proof}

\begin{remark}\label{rem:cursterline}
As in (ii) of Theorem \ref{thm:reflection}, special kind of boundary lines on zero mean curvature surfaces in several ambient spaces appear as cluster point sets of conformal mappings. For example,  points of minimal graphs in $\mathbb{E}^3$ of the form $t=\varphi(x,y)$ on which the function $\varphi$ diverges to $\pm \infty$, and lightlike line segments on boundary of maximal surfaces in $\mathbb{L}^3$ can be also written as cluster point sets of conformal mappings, see \cite{AF1} for more details. In particular, a reflection principle for lightlike line segments was proved in \cite{AF2}. 

\end{remark}

As a special case of Theorem \ref{thm:reflection}, we can consider the following more specific boundary conditions.

\begin{corollary}\label{cor:hline}
Under the same assumptions as in Theorem \ref{thm:reflection}, suppose the isotropic line segment $L$ connects two parallel horizontal straight line segments $l_i$ $(i=1, 2)$ on $\partial{S}$. Then the surface $S$ can be extended real analytically across $L$ via the symmetry with respect to the parallel line to $l_i$ $(i=1, 2)$ passing through the midpoint of $L$.
\end{corollary}

\begin{proof} 
By the assumption, $h(-\varepsilon, \varepsilon )$ in the proof of Theorem \ref{thm:reflection} is a line segment and we may assume that this line segment is on the $x$-axis. Then
$h\circ \Pi (-r,\pi-\theta) =\overline{h(re^{i\theta})}$ holds by \eqref{eq:h} and hence $X=(h,t)$ satisfies 
\begin{align*}
 X\circ \Pi(-r, \pi-\theta)  &=\left( \overline{h(re^{i\theta})}, a+b - t\circ\Pi(r,\theta) \right),
\end{align*}
which implies that the extension $X\circ \Pi$ is invariant under the symmetry with respect to the parallel line to $x$-axis passing through the midpoint of $L$.
\end{proof}

For a zero mean curvature surface $S$ parametrized as \eqref{eq:w-formula} with Weierstrass data $(F,G)$, the surface $X^*$ with Weierstrass data $(iF,G)$ is called the {\it conjugate surface} of $S$, see \cite[p.424]{Strubecker}, \cite[p.~238]{Sachs} and \cite{Sato}. 
Let $X=(h,t)\colon \mathbb{H} \to \mathbb{I}^3 \simeq \mathbb{C}\times \mathbb{R}$ be a conformal parametrization of $S$. Then $X^*:=-(h^*,t^*)$ is a conformal parametrization of $S^*$, which are formed by conjugate harmonic functions.
In the end of this section, we mention that isotropic lines of $S$ correspond to points of $S^*$ on which $t^*$ diverges to $\pm \infty$ as follows.
\begin{corollary}\label{cor:conjugate}
Under the same assumption as in Theorem \ref{thm:reflection}, 
\begin{equation}\label{eq:conjugate_behavior}
\displaystyle \lim_{w\to w_0}\lvert t^*(w)\rvert=\infty.
\end{equation}
\end{corollary}
\noindent Here, we remark that since $h^*=ih$ the $(x,y)$ coordinates of $S^*$ are essentially only the $90^\circ$-rotation of those of $S$.

\begin{proof}
We use the formulation of the proof of Theorem \ref{thm:reflection}. We can easily see that $t$ is written as
\begin{equation}\label{eq:t2}
t(w)= a+b+\frac{a-b}{\pi}\arg{w}-\frac{a}{\pi}\arg{(-\varepsilon -w)} +\frac{b}{\pi}\arg{(\varepsilon+w)}+P_W,
\end{equation}
where $P_W$ is the Poisson integral of $W:=(1-\chi_{(-\varepsilon, \varepsilon)})\hat{t}$ and $\chi_{(-\varepsilon, \varepsilon)}$ is the characteristic function on $(-\varepsilon, \varepsilon)$. 
By \eqref{eq:t2} and the fact that the conjugate function $\arg^*{(w)}$ is $-\log{|w|}$, we obtain
\[
t^*(w) = -\frac{a-b}{\pi}\log{\lvert w \rvert} + \frac{a}{\pi} \log{\lvert-\varepsilon -w\rvert} - \frac{b}{\pi}\log{\lvert \varepsilon+w\rvert} +P^*_W +c,
\]
for some constant $c$. Since $P_W\lvert_{(-\varepsilon, \varepsilon)} =0$, it follows that $\lim_{r\to 0}P^*_W(re^{i\theta})$ is a constant and hence $\displaystyle \lim_{r\to 0}\lvert t^*(re^{i\theta})\rvert=\infty$. Therefore, we obtain the desired result.
\end{proof}

\section{Examples}
By Theorem \ref{thm:reflection}, we can construct zero mean curvature surfaces in $\mathbb{I}^3$ with isotropic line segments.

\begin{example}[Isotropic helicoid and catenoid]
If we take the Weierstrass data $(F,G)=\left(1,\frac{1}{2\pi i w}\right)$ defined on $\mathbb{H}$, then by using \eqref{eq:w-formula} we have the helicoid 
\[
X(re^{i\theta}) = \left( r\cos{\theta},r\sin{\theta}, \frac{\theta}{\pi} \right),\quad w=re^{i\theta} \in \mathbb{H}.
\]

Using the notations in Section \ref{subsec:vline} and by Corollary \ref{cor:hline}, $X$ can be extended to $X\circ \Pi (r, \theta) =X(re^{i\theta})$ defined on $\mathbb{R}\times (0, \pi)$ across the isotropic line segment $L=\{ X\circ \Pi(0, \theta)\in \mathbb{I}^3\mid \theta \in [0,\pi]\}$.
Moreover, by Corollary \ref{cor:reflection_nonvertical}, we can extend $X\circ \Pi$ across the horizontal lines on the boundary. Repeating this reflection, we have the entire part of the singly periodic helicoid in $\mathbb{I}^3$. See Figure \ref{Fig:Helicoids}.

\begin{figure}[htbp]
 \vspace{3.5cm}
\hspace*{-8ex}
    \begin{tabular}{cc}
     \hspace*{-3ex}

      \begin{minipage}[t]{0.55\hsize}
        \centering
                \vspace{-3.2cm}
                \hspace*{-5ex}
        \includegraphics[keepaspectratio, scale=0.36]{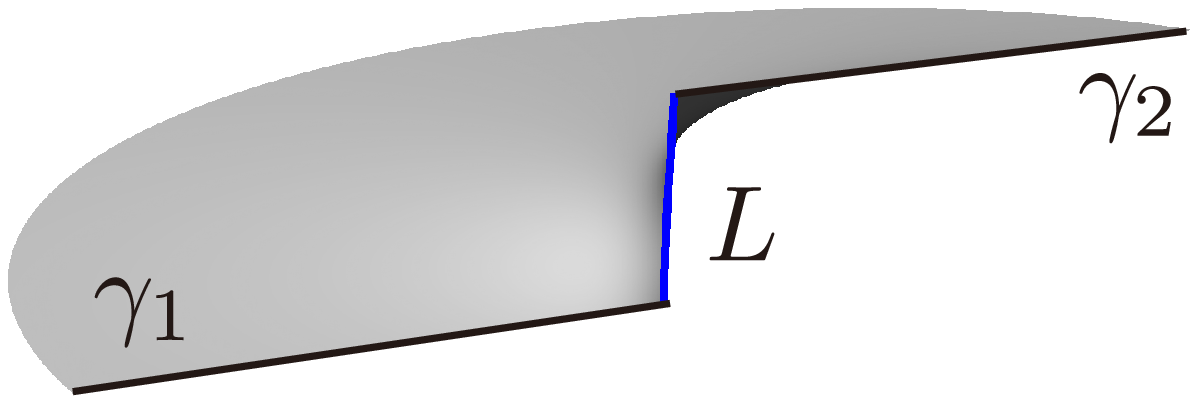}
              \end{minipage} 
                 \hspace*{-19ex}
        \begin{minipage}[t]{0.55\hsize}
        \centering
            \vspace{-3.5cm}
        \includegraphics[keepaspectratio, scale=0.36]{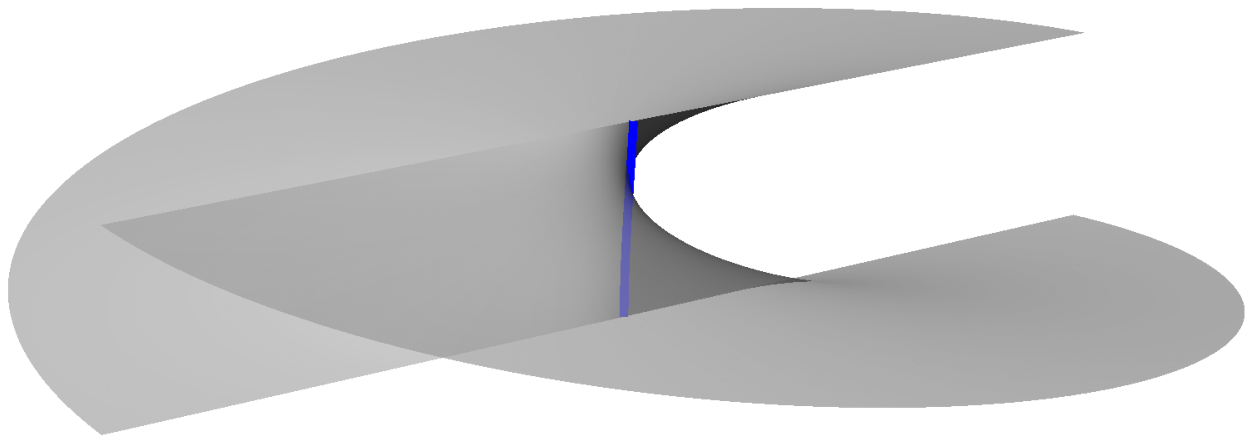}
              \end{minipage} 
           \hspace*{-9ex}
      \begin{minipage}[t]{0.50\hsize}
        \centering
        \vspace{-3.5cm}
        \includegraphics[keepaspectratio, scale=0.36]{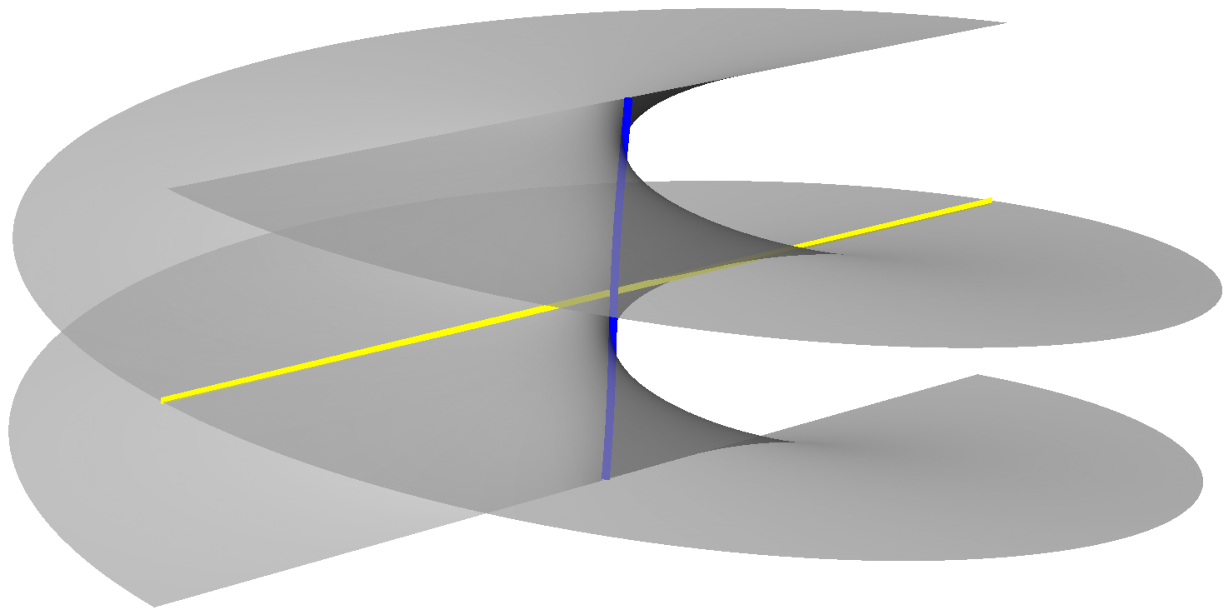}
            \end{minipage} 
    \end{tabular}
     \vspace{-1.5cm}
       \caption{The left one is a part of the helicoid in $\mathbb{I}^3$ with an isotropic line segment $L$, the center one is its reflection across $L$, and the right one is the surface after the reflection of a horizontal line. For the notations of $\gamma_1$ and $\gamma_2$, see Theorem \ref{thm:reflection_Intro}.}
       \label{Fig:Helicoids}
  \end{figure}
\end{example} 

The conjugate surface of $X$ is written as
\[
X^*(re^{i\theta}) = \left( r\sin{\theta}, -r\cos{\theta},\frac{1}{\pi}\log{r} \right),\quad w=re^{i\theta} \in \mathbb{H},
\]
which is the half-piece of a rotational zero mean curvature surface called the {\it isotropic catenoid}. 
By Corollary \ref{cor:conjugate}, $L$ on $X$ corresponds to the limit $\lim_{r\to 0}X^*(re^{i\theta})=(0,0, -\infty)$. Moreover, since horizontal straight lines of $X$ also correspond to curves on $X^*(\partial{\mathbb{H}})$ in the $yt$-plane, we can extend $X^*$ via the reflection with respect to this vertical plane. See Figure \ref{Fig:iCat}. 
\begin{figure}[htbp]
 \vspace{3.6cm}
\hspace{-8ex}
    \begin{tabular}{cc}
      \begin{minipage}[t]{0.55\hsize}
        \centering
                \vspace{-3.5cm}
        \includegraphics[keepaspectratio, scale=0.5]{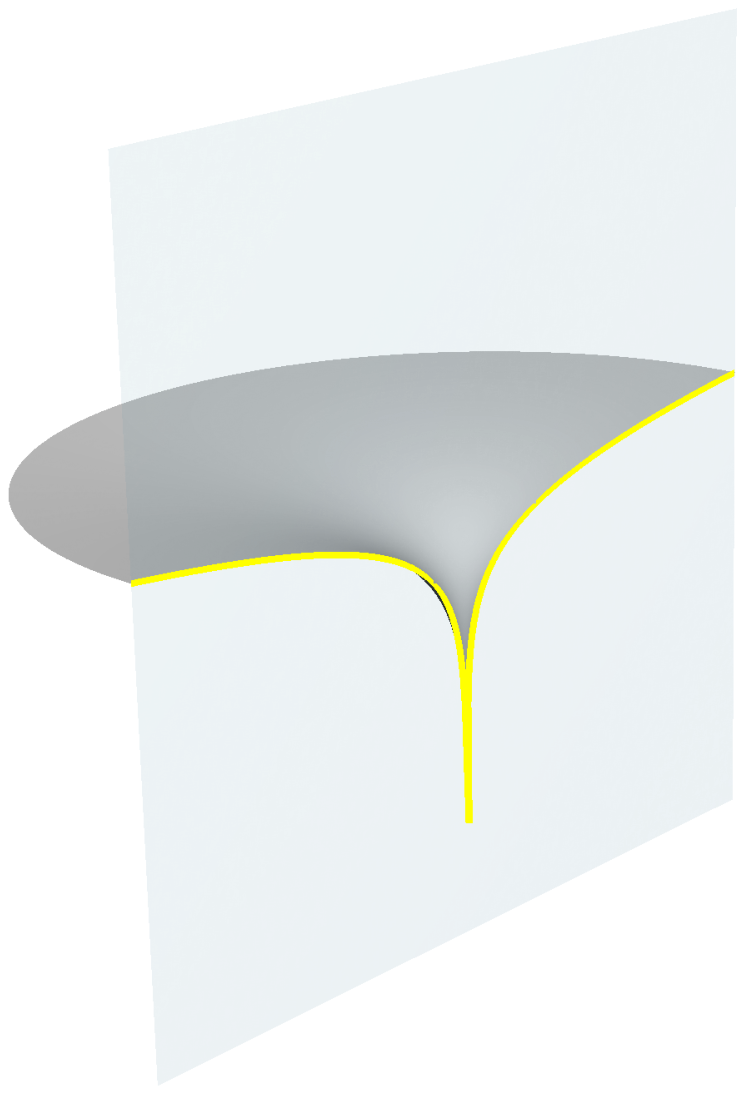}
              \end{minipage} 
                 \hspace{-7ex}
        \begin{minipage}[t]{0.55\hsize}
        \centering
            \vspace{-3.7cm}
        \includegraphics[keepaspectratio, scale=0.6]{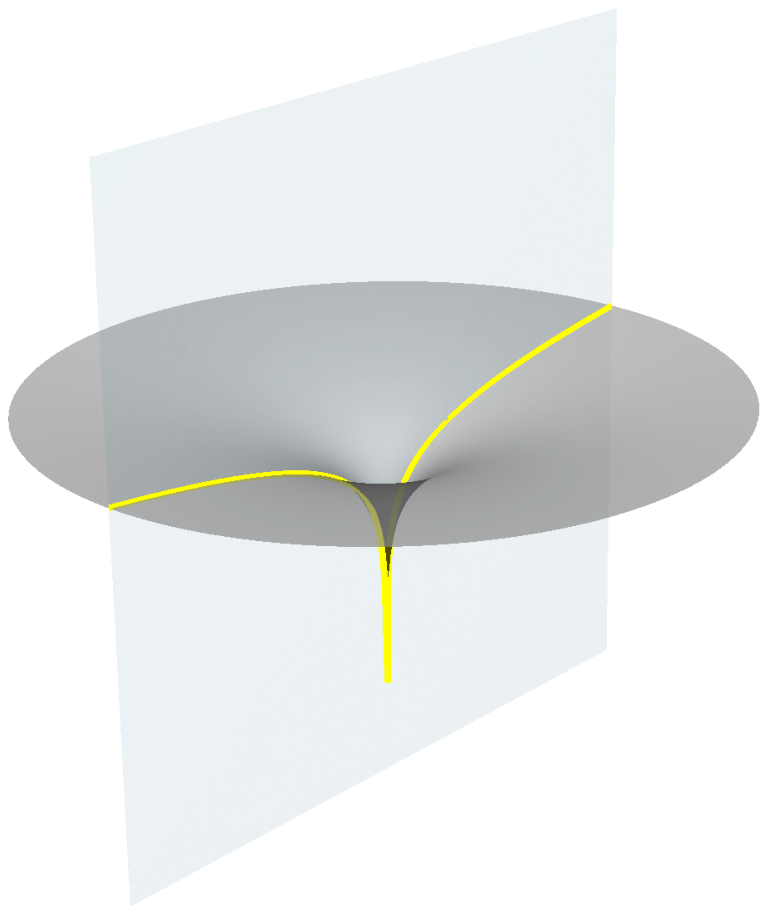}
              \end{minipage} 
    \end{tabular}
      \caption{The isotropic catenoid (left) and its analytic extension (right).}
       \label{Fig:iCat}
  \end{figure}

\begin{example}[Isotropic Schwarz D-type surface] \label{ex:Schwarz}
For an integer $n \geq 2$, it is known that the Schwarz-Christoffel mapping $f\colon \mathbb{D} \to \mathbb{C}$ defined by
\begin{equation*}
	f(w)=\int^w_0 \frac{1}{(1-w^{2n})^{\frac{1}{n} } } dw
\end{equation*}
maps the unit disk $\mathbb{D}$ conformally to a regular $2n$-gon $\Omega$ (see \cite[Chapter 6, Section 2.2]{Ahl1}). The mapping $f$ extends homeomorphically to $f\colon \overline{\mathbb{D}} \to \overline{\Omega}$ by the Carath\'eodory theorem, and $w=e^{k \pi i/n}\ (k=1,2,\ldots,2n)$ correspond to the vertices of $\Omega$. On the other hand, by using the equations
\begin{equation*}
	f(e^{\pi i/n}w)=e^{\pi i/n}f(w), \quad f(\overline{w})=\overline{f(w)},
\end{equation*}
we can see that the boundary points
\begin{equation*}
	w_k := e^{ \frac{k \pi i}{n} - \frac{\pi i}{2n} }\ \ (k=1,2,\ldots,2n)
\end{equation*}
correspond to the midpoints of edges of $\Omega$.

Let $I_k$ be the shortest arc of $\partial \mathbb{D}$ joining $w_k$ and $w_{k+1}$ ($k=1,2,\ldots,2n$) where $w_{2n+1}:=w_1$, and let
\begin{eqnarray*}
	\hat{t}(w):=\left\{ 
			\begin{array}{ll}
				1 & \ \ (w \in I_{2k-1},\ k=1,2,\ldots ,n )\\[1ex]
				0 & \ \ (w\in I_{2k},\ k=1,2,\ldots , n )
			\end{array} 
		\right. .
\end{eqnarray*}
The Poisson integral of $\hat{t}$ can be easily computed, and we have
\begin{eqnarray*}
	t(w)&=&\frac{1}{2\pi}\int_0^{2\pi}\frac{1-\lvert w \rvert^2}{\lvert e^{is}-w \rvert^2}\ \hat{t}(e^{is})ds \\[1ex]
			&=&\sum_{k=1}^n \frac{1}{\pi} \arg \left( \frac{w_{2k}-w}{w_{2k-1}-w} \right) - \frac{1}{2}.
\end{eqnarray*}
Then $X:=(f,t)\colon \mathbb{D} \to \mathbb{I}^3$ is a zero mean curvature surface with $2n$-isotropic lines in its boundary, and by construction, each of the isotropic lines connects two parallel horizontal line segments on the boundary. Therefore, Corollary \ref{cor:hline} is applicable, and $S:=X(\mathbb{D})$ extends real analytically across each isotropic line $L$ via the $180^\circ$-degree rotation with respect to the straight line parallel to the edge of $\Omega$ passing through the midpoint of $L$ (see Figure \ref{Fig:polygon1}).

\begin{figure}[htbp]
    \vspace{38ex}
    \begin{tabular}{cc}
      \hspace{-5ex} 
      \begin{minipage}[t]{0.4\hsize}
        \centering \vspace{-40ex}
        \includegraphics[keepaspectratio, scale=0.40]{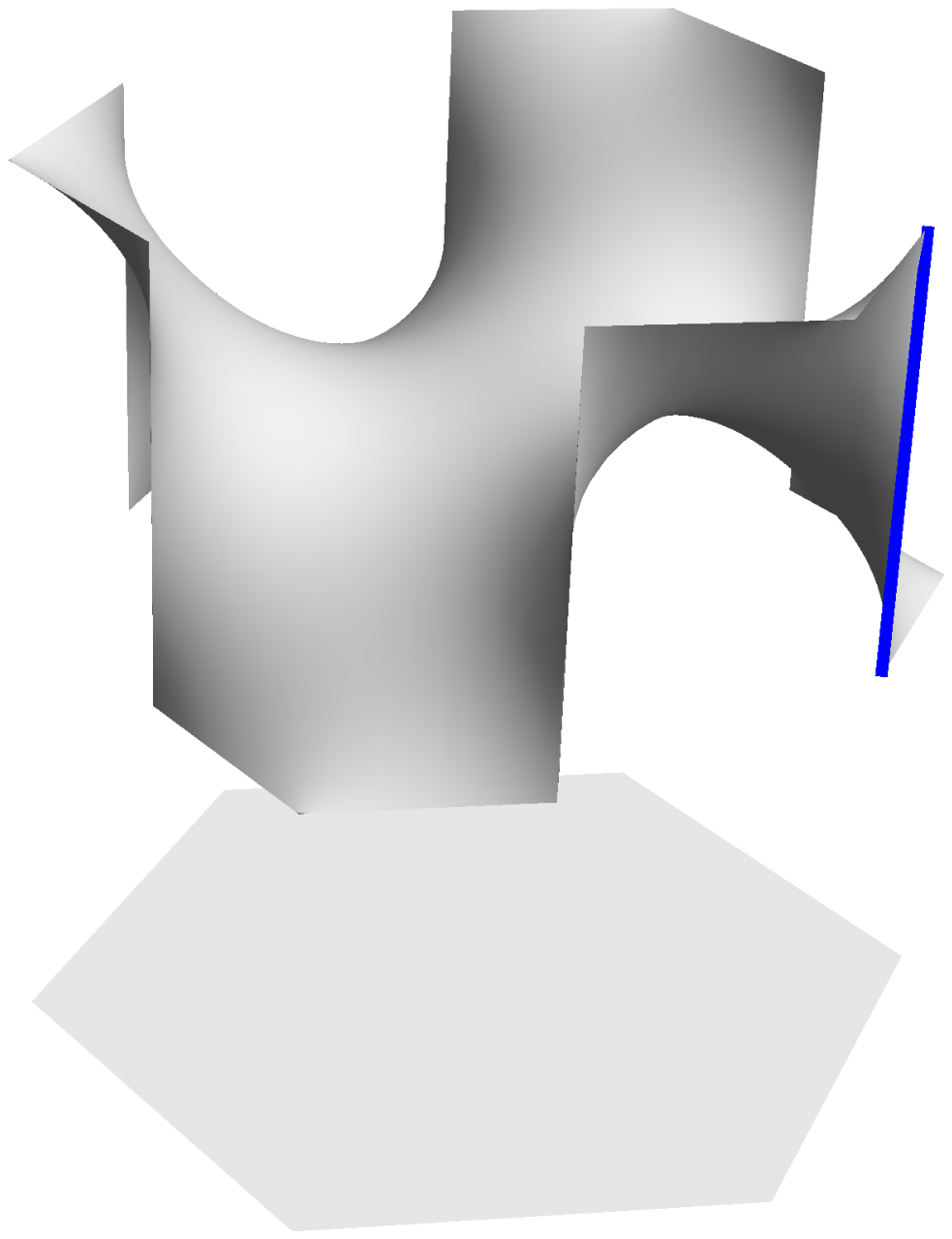}  \label{zu1}
      \end{minipage} &
      \hspace{-5ex} 
      \begin{minipage}[t]{0.6\hsize}
     \vspace{-40ex} 
        \centering
        \includegraphics[keepaspectratio, scale=0.60]{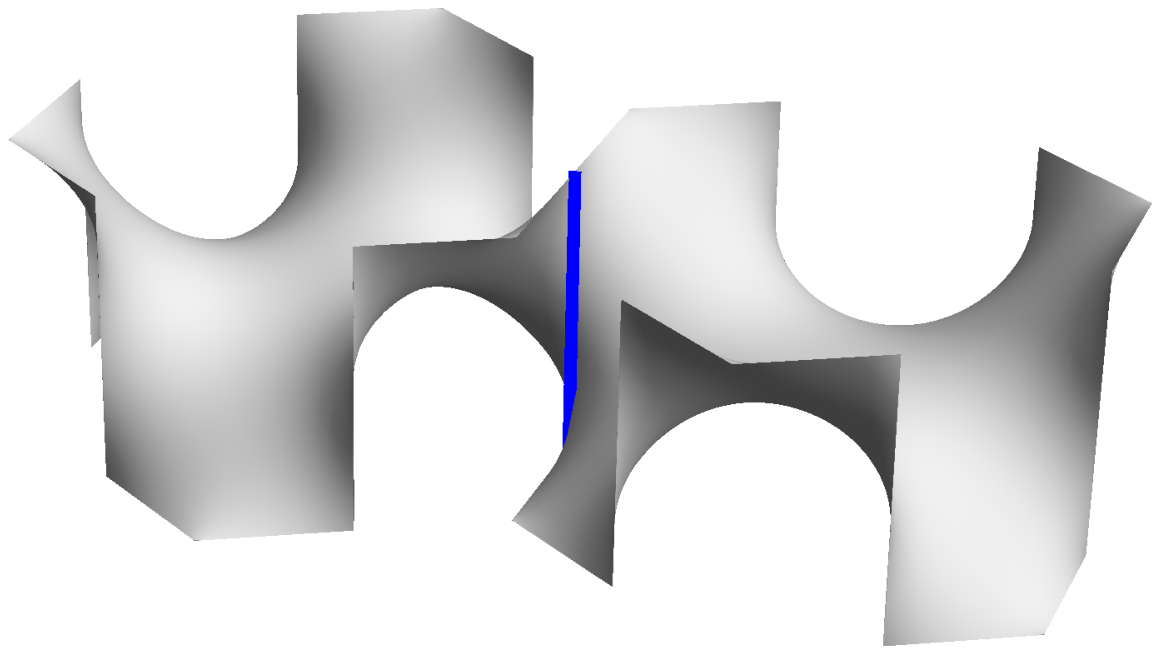}  \label{zu2}
      \end{minipage}
    \end{tabular}
    \vspace*{-1.8cm}
    \caption{Case $n=3$ in Example \ref{ex:Schwarz}: a zero mean curvature surface whose boundary consists of horizontal lines and isotropic lines (left) and its extension across an isotropic line (right).}\label{Fig:polygon1}
\end{figure}

\noindent
In particular, if $n=2$ (i.e. $\Omega$ is a square), we can obtain a triply periodic zero mean curvature surface in $\mathbb{I}^3$ which is analogous to Schwarz' D minimal surface in $\mathbb{E}^3$ (cf.~\cite{DHS}) with isotropic lines by iterating reflections of $S$ (see Figure \ref{Fig:polygon2}).

\begin{figure}[htbp]
\vspace{26ex}
    \begin{tabular}{cc}
    \hspace{-6em}
      \begin{minipage}[t]{0.4\hsize}
        \centering \vspace{-23ex} 
        \includegraphics[keepaspectratio, scale=0.36]{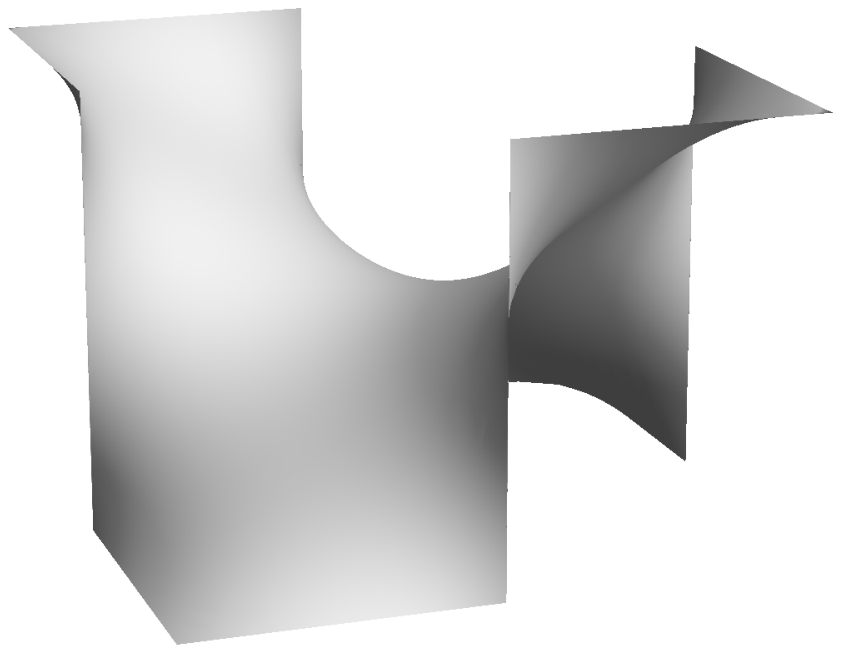}  \label{zu1}
      \end{minipage} &
      \begin{minipage}[t]{0.6\hsize}
         \vspace{-31ex} 
        \centering \hspace{-19ex}
        \includegraphics[keepaspectratio, scale=0.38]{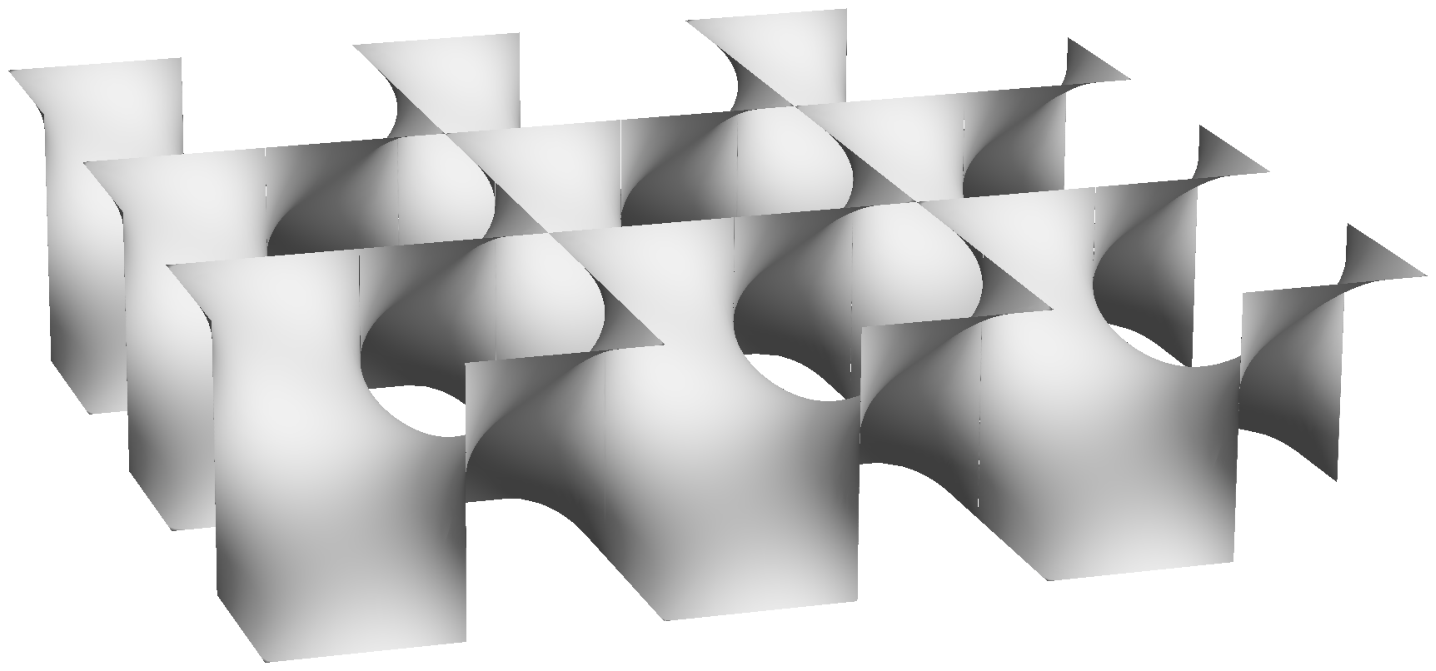}  \label{zu2}
      \end{minipage}
      \begin{minipage}[t]{0.6\hsize}
        \vspace{-30ex} 
        \centering \hspace{-39ex}
        \includegraphics[keepaspectratio, scale=0.5]{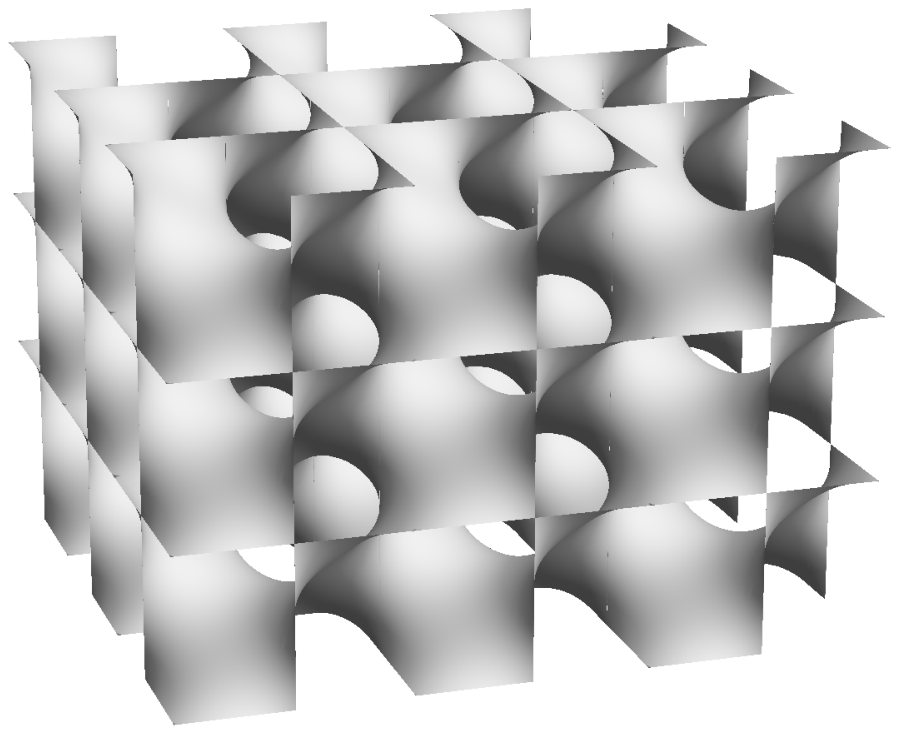}  \label{zu2}
      \end{minipage}
    \end{tabular}
    \vspace*{-1.2cm}
    \caption{Case $n=2$ in Example \ref{ex:Schwarz}: construction of a triply periodic zero mean curvature surface.}\label{Fig:polygon2}
\end{figure}
\end{example}


\begin{acknowledgement}
The authors would like to express their gratitude to the referees for their careful readings of the submitted version
of the manuscript and fruitful comments and suggestions.
\end{acknowledgement}




\end{document}